\documentclass{article}

\usepackage{arxiv}

\usepackage[utf8]{inputenc} 
\usepackage[T1]{fontenc}    
\usepackage{hyperref}       
\usepackage{url}            
\usepackage{booktabs}       
\usepackage{amsfonts, amsthm}       
\usepackage{nicefrac}       
\usepackage{microtype}      
\usepackage{lipsum}		
\usepackage{graphicx}
\usepackage{natbib}
\usepackage{doi}
\usepackage{amsmath}

\usepackage{algorithm}
\usepackage{algpseudocode}

\usepackage{tabularx} 
\usepackage{graphicx}
\usepackage{caption}
\usepackage{subcaption}
\usepackage[dvipsnames]{xcolor}
\newcommand{\w}{\mathbf}
\newcommand{\R}{\mathbb{R}}

\newtheorem{theorem}{Theorem}[section]

\newtheorem{definition}[theorem]{Definition}
\newtheorem{assumption}[theorem]{Assumption}

\title{Simba: A Scalable Bilevel Preconditioned Gradient Method for Fast Evasion of Flat Areas and Saddle Points}

\author{ Nick Tsipinakis \\ 
	Department of Mathematics\\
	UniDistance Suisse\\
	Brig, Switzerland \\
	\texttt{nikolaos.tsipinakis@unidistance.ch} \\
	\And
        Panos Parpas \\
	Department of Computing\\
	Imperial College London\\
	London, UK \\
	\texttt{panos.parpas@imperial.ac.uk} \\
}

\renewcommand{\shorttitle}{\textit{arXiv} Template}

\begin{document}
\maketitle

\begin{abstract}

The convergence behaviour of first-order methods can be severely slowed down when applied to high-dimensional non-convex functions due to the presence of saddle points. If, additionally, the saddles are surrounded by large plateaus, it is highly likely that the first-order methods will converge to sub-optimal solutions. In machine learning applications, sub-optimal solutions mean poor generalization performance. They are also related to the issue of hyper-parameter tuning, since, in the pursuit of solutions that yield lower errors, a tremendous amount of time is required on selecting the hyper-parameters appropriately. A natural way to tackle the limitations of first-order methods is to employ the Hessian information. However, methods that incorporate the Hessian do not scale or, if they do, they are very slow for modern applications. Here, we propose Simba, a scalable preconditioned gradient method, to address the main limitations of the first-order methods. The method is very simple to implement. It maintains a single precondition matrix that it is constructed as the outer product of the moving average of the gradients. To significantly reduce the computational cost of forming and inverting the preconditioner, we draw links with the multilevel optimization methods. These links enables us to construct preconditioners in a randomized manner. Our numerical experiments verify the scalability of Simba as well as its efficacy near saddles and flat areas. Further, we demonstrate that Simba offers a satisfactory generalization performance on standard benchmark residual networks. We also analyze Simba and show its linear convergence rate for strongly convex functions. 
\end{abstract}


\section{Introduction}

We focus on solving the following optimization problem, commonly referred to as empirical risk minimization:
\begin{equation} \label{fine model}
    \min_{\w{x} \in \mathbb{R}^n} f(\w{x}) := \frac{1}{m} \sum_{i=1}^m f_i(\w{x}).
\end{equation}
We assume that the model parameters are available in a \textit{decoupled} form, that is, $\w{x} = \{\w{x}^l : l = 1, \ldots L \}$, for some positive integer $L$. The decoupled parameter setting naturally emerges for the training of a deep neural network (DNN) where $L$ corresponds to the number of layers. In large-scale optimization scenarios, stochastic gradient descent method (SGD) and its variants have become a standard tool due to their simplicity and low per-iteration and memory costs. While easy to implement, they are nevertheless accompanied with important shortcomings. Below, we discuss these limitations and argue that they mainly stem from the absence of the Hessian matrix (or an appropriate approximation that leverages its structural properties) in their iterative scheme.

\textbf{Shortfall I: Slow evasion of saddle points.} 
The prevalence of saddle points in high dimensional neural networks has been examined in the past and it has been shown that the number of saddles increases exponentially with the dimension $n$ \cite{dauphin2014identifying}. Consequently, developing algorithms that are capable of escaping saddle points is a primary goal when it comes to the training of DNNs. While stochastic first-order methods have theoretical guarantees of escaping saddle points almost always \cite{panageas2019first}, it remains unclear whether they can achieve this efficiently in practice. As a result, given the significantly larger number of saddle points compared to local minima in high-dimensional settings, it is likely that stochastic first-order methods will converge to a saddle point rather than the desired local minimum.
The most natural way to ensure convergence to a local minimum in the presence of saddles is by incorporating the Hessian matrix to scale the gradient vector. The Cubic Newton \cite{nesterov2018lectures} is a key method in non-convex optimization known to be able to escape saddles in a single iteration resulting in rapid convergence to a local minimum. The Cubic Newton method achieves this using the inverse of a regularized version of the Hessian matrix to premultiply the gradient vector. However, the method has expensive iterates and does not scale in high dimensions. Nevertheless, its theory indicates that methods employing accurate approximations of the regularized Hessian are expected to effectively evade saddle points and converge rapidly.

\textbf{Shortfall II: Slow convergence rate near flat areas and the vanishing gradient issue.} 
The convergence behaviour of first-order methods can be significantly affected near local minima that are surrounded by plateaus \cite{MR2061575, MR2142598}. For instance, this occurs, when its Hessian matrix has most of the eigenvalues equal to zero. Therefore,  to effectively navigate on such a flat landscape is to employ the inverse Hessian matrix. For instance, the Newton method pre-multiplies the gradient with the inverse Hessian to perform a local change of coordinates which significantly accelerates the convergence rate. Intuitively, one can expect a similar convergence behaviour by replacing the Hessian matrix with a preconditioned matrix that retains its structure. On the other hand,
the vanishing gradient issue is commonly observed in the training of DNNs for which activation functions that range in $[0,1]$ are required. In such cases, the gradient value computed through the back-propagation decreases exponentially as the number of layers $L$ of the network increase yielding no progress for the SGD. Hence, similar to addressing the slow-convergence issue near flat areas, employing a preconditioner is expected to mitigate the vanishing gradient issue and thus accelerate the convergence of the SGD method. 

To tackle the above shortfalls, the diagonal (first- or second-order) methods were introduced \cite{duchi2011adaptive, tieleman2012lecture, zeiler2012adadelta, kingma2014adam, zaheer2018adaptive, yao2021adahessian, ma2020apollo, jahani2021doubly, liu2023sophia}. These algorithms aim to improve the behavior of the SGD near saddles or flat areas by preconditioning the gradient with a diagonal vector that mimics the hessian matrix. However, the Hessian matrices of deep neural networks need not be sparse, which means that a diagonal approximation of the Hessian may be inadequate. This raises concerns about the ability of diagonal methods to converge to a local minimum. A more promising class of algorithms consists of those preconditioned gradient methods that maintain a matrix instead of the potentially poor diagonal approximation. As noted earlier, the Cubic Newton method can effectively address the aforementioned shortfalls but it does not scale in large-scale settings. Other methods that explicitly use or exploit the Hessian matrix, such as Newton or Quasi-Newton methods \cite{MR2142598, broyden1967quasi}, or those based on randomization and sketching \cite{erdogdu2015convergence, pilanci2017newton, xu2016sub, xu2020second} still suffer from the same limitations as the Cubic Newton. A potentially promising way to address the storage and computational limitations of the Newton type methods is to employ precondition matrices based on the first-order information \cite{duchi2011adaptive, gupta2018shampoo}. However, to the best of our knowledge, the performance of such methods near saddles or flat areas is yet to be examined. Another possible solution is to utilize the multilevel framework in optimization which significantly reduces the cost of forming and inverting the Hessian matrix \cite{tsipinakis2021multilevel, tsipinakis2023multilevel}. But, it remains unclear whether the multilevel Newton-type methods can be efficiently applied to the training of modern deep neural networks.


\begin{algorithm}[t]
\caption{Simba. Proposed default hyper-parameters: $\w{R}_k$ as in Definition \ref{def: P}, $t_k = 10^{-2}, b=0.9, r=20, n_\ell = 0.5n, m = 10^{-8}$.}\label{alg: alg}
\begin{algorithmic}
\Require $t_k$ (step-size), $\beta \in (0,1)$ (momentum), $\w{x}_0$ - (initial point), $r$  (No of eigenvalues), $n_\ell$ (coarse model dimensions), $m$ (eigenvalue correction).
\State Set $\w{G}_0 = 0$
\For{$k = 1, 2, \ldots$}
\State $\w{G}_k = \beta \w{G}_{k-1} +  \nabla f(\w{x}_{k-1})$ (EMA of gradients)
\State $\w{G}_{\ell, k} = \w{R}_k \w{G}_k$ (coarse EMA of gradients)
\State $\w{Q}_{\ell,k} = \w{G}_{\ell, k} \w{G}_{\ell, k}^T$ (EMA-based preconditioner)
\State $\left[ \w{U}_{r+1} \w{\Lambda}_{r+1} \right] = \text{T-SVD}(\w{Q}_{\ell,k})$ by \cite{MR2806637}
\State       $ \left[  \w{\Lambda}_{r+1} \right]_{i, m} = \begin{cases} \w{\Lambda}_{i}, & \w{\Lambda}_{i} \geq m \\
            m, & \text{otherwise}
            \end{cases}, \quad $ and form $\w{\Lambda}_{r+1}^m $
\State $\w{Q}_{\ell,k}^{-\frac{1}{2}} = \left[\w{\Lambda}_{r+1} \right]_{r+1, m}^{-\frac{1}{2}}  \w{I}_{n_\ell} +  \w{U}_r \left(  \left[\w{\Lambda}_{r+1}^m \right]^{-\frac{1}{2}} - \left[\w{\Lambda}_{r+1} \right]_{r+1, m}^{-\frac{1}{2}} \w{I}_r \right) \w{U}_r^T$
\State $\w{x}_{k+1} = \w{x}_k - t_k \w{P}_k \w{Q}_{\ell,k}^{-\frac{1}{2}} \w{R}_k \w{G}_k$
\EndFor
\end{algorithmic}
\end{algorithm}

In this work, we introduce Simba, a \textbf{S}calable \textbf{I}terative \textbf{M}inimization \textbf{B}ilevel \textbf{A}lgorithm that addresses the above shortfalls and scales to high dimensions. The method is very simple to implement. It maintains a precondition matrix where its inverse square root is used to premultiply the moving average of the gradient. In particular, the precondition matrix is defined as the outer product of the exponential moving average (EMA) of the gradients. We explore the link between the multilevel optimization methods and the preconditioned gradient methods which enables us to construct preconditioners in a randomized manner. By leveraging this connection, we significantly reduce the cost and memory requirements of these methods to make them more efficient in time for large-scale settings. We propose performing a Truncated Singular Value decomposition (T-SVD) before constructing the inverse square root preconditioner. Here, we retain only the first few, say $r$, eigenvalues and set the remaining eigenvalues equal to the $r^{\text{th}}$ eigenvalue. This approach aims to construct meaningful preconditioners that do not allow for zero eigenvalues, and thus expecting a faster escape rate from saddles or flat areas.  
Our algorithm is inspired by SigmaSVD \cite{tsipinakis2023multilevel}, Shampoo \cite{gupta2018shampoo} and SGD with momentum \cite{bubeck2015convex}. Simba is presented in Algorithm \ref{alg: alg}.  
A Pytorch implementation of Simba is available at \url{https://github.com/Ntsip/simba}.

\subsection{Related Work}
The works most closely related to ours is Shampoo \cite{gupta2018shampoo} and AdaGrad \cite{duchi2011adaptive}. However, both methods have important differences to our approach which are discussed below. AdaGrad is an innovative adaptive first-order method which performs elementwise division of the gradient with the square root of the accumulated squared gradient. Adagrad is known to be suitable for sparse settings. However, the version of the method that retains a full preconditioner matrix is rarely used in practice due to the increased computational cost. Moreover, Shampoo, unlike our approach, maintains multiple preconditioners, one for each tensor dimension (i.e., two for matrices while we always use one). In addition, it computes a full SVD for each preconditioner. Further, all preconditioners lie in the original space dimensions. Given the previous considerations, Shampoo becomes computationally expensive when applied to very large DNNs. In contrast, our algorithm addresses the computational issues of Shampoo. Another difference between Shampoo and our approach is the way the preconditioners are defined.  Shampoo defines the preconditioners as a sum of the outer product of gradients, whereas we employ the exponential moving average. In addition, the performance of Shampoo around saddles and flat areas is yet to be examined. We conjecture that Shampoo may not effectively improve the behaviour of the first-order methods near plateaus due to the monotonic increase of the eigenvalues in the preconditioners.

Hessian-based preconditioned methods have been also proposed to improve convergence of the optimization methods near saddle point \cite{reddi2018generic, dauphin2014identifying, o2020low}. However, these methods rely on second-order information which can be prohibitively expensive for very large DNNs. Therefore, they are more suitable to apply in conjunction with a fast first-order method and employ the Hessian approximation when first-order methods slow down. However, our method is efficient from the beginning of the training. A notable second-order optimization method is K-FAC \cite{martens2015optimizing}. K-FAC approximates the Fisher information matrix of each layer using the Kronecker product of two smaller matrices that are easier to handle. However, K-FAC is specifically designed for the training of generative (deep) models while our method is general and applicable to any stochastic optimization regime.

As far as the adaptive diagonal methods are concerned, Adam and RMSprop were introduced to alleviate the limitations of AdaGrad. RMSprop scales the gradient using an averaging of the squared gradient. On the other hand, Adam scales the EMA of the gradient by the EMA of the squared gradient. Adam iterations has been shown to significantly improve the performance of first-order diagonal methods in the optimization of DNNs. Another method, named Yogi, has similar updates as AdaGrad but allows for the accumulated squared root not be monotonically increasing. To provide EMA of the gradient with a more informative scaling, several approaches have been proposed that replace the squared gradient with an approximation of the diagonal of the Hessian matrix \cite{jahani2021doubly, yao2021adahessian, ma2020apollo}. AdaHessian \cite{yao2021adahessian} approximates the Hessian matrix by its diagonal matrix using the Hutchinson's method while Apollo \cite{ma2020apollo} based on the variational technique using the weak secant equation. Further, OASIS \cite{jahani2021doubly} is an adaptive modification of AdaHessian that automatically adjusts learning rates. A recent work called Sophia \cite{liu2023sophia} modifies AdaHessian direction by clipping it by a scalar. The authors also provide alternatives on the diagonal Hessian matrix approximation which they suggest to compute every $10$ iterations to reduce the computational costs. However, it is still unclear whether the diagonal methods offer satisfactory performance near saddles points and flat areas, or if they effectively tackle the vanishing gradient issue in practical applications.

\subsection{Contributions}

The main contribution of this paper is the development of a scalable method that addresses the aforementioned shortfalls by empirical evidence. The method is scalable for training DNNs as it maintains only one preconditioned matrix at each iteration which lies in the subspace (coarse level). Thus, we significantly reduce the cost of forming and computing the SVD compared to Shampoo. The fact that we use a randomized T-SVD and requiring the most informative eigenvalues further reduces the total computational cost. The numerical experiments demonstrate that
Simba can have cheaper iterates than AdaHessian and Shampoo in large-scale settings and deep architectures. In particular, in these scenarios, the wall-clock time of our method can be as much as $25$ times less than Shampoo and $2$ times less than AdaHessian. We, in addition, illustrate that Simba offers comparable, if not better, generalization errors against the state-of-the-art methods in standard benchmark ResNet architectures.

Further, the numerical experiments show that our method has between two and three times more expensive iterates than Adam. This is expected due to outer products and the randomized T-SVD. However, this is a reasonable price to pay to improve the escape rate from saddles and flat areas of the existing preconditioned gradient methods. Therefore, Simba is suitable in problems where diagonal methods suffer from at least of one of the previous shortfalls. In such cases, we demonstrate that diagonal methods require much larger wall-clock time to reach an error as low as our method achieves. Hence, we argue that, besides the encouraging preliminary empirical results achieved by Simba, this paper highlights the limitations of diagonal methods, which are likely to get trapped near saddle points and thus return sub-optimal solution, particularly in the presence of large plateaus. This emphasizes the need to develop scalable algorithms that utilize more sophisticated preconditioners instead of relying on poor diagonal approximations of the Hessian matrix. Simba is accompanied with a simple convergence analysis assuming strongly convex functions.

\section{Description of the Algorithm}

In this section, we begin by discussing the standard Newton-type multilevel method and its main components, i.e., the hierarchy of coarse models and the linear operators that will used to transfer information from coarse to fine model and vice versa. Then, as all the necessary ingredients are in place, we present the proposed algorithm. Even though, we assume a bilevel hierarchy, extending Simba to several level is straightforward.

\subsection{Background Knowledge - Multilevel Methods}

We will be using the multilevel terminology to denote $f$ as the \emph{fine} model and also assume that a \emph{coarse} model that lies in lower dimensions is available. In particular, the coarse model is a maping $F: \R^{n_\ell} \rightarrow \R$, where $n_\ell < n$. The subscript $\ell$ will be used to denote quantities, i.e., vectors and scalars, that belong to the coarse level. We assume that both fine and coarse models are bounded from below so that a minimizer exists.

In order to reduce the computational complexity of Newton type methods, the standard multilevel method attempts to solve \eqref{fine model} by minimizing the coarse model to obtain search directions. This process is described as follows. First, one needs to have access to linear prolongation and restriction operators
to transfer information from the fine to coarse model and vice versa. We denote $\w{P} \in \R^{n \times n_\ell}$ be the prolongation and $\w{R} \in \R^{n_\ell \times n}$ the restriction operator and assume that they are full rank and $\w{P} = \w{R}^T$. Given a point $\w{x}_k$ at iteration $k$, we move to the coarse level with initialization point $\w{y}_0 = \w{R} \w{x}_k$. Subsequently, we minimize the coarse model $F$ to obtain $\w{y}^*$ and construct a search direction by:
\begin{equation} \label{eq: coarse model}
    \w{d}_k := \w{P}_k (\w{y}^* - \w{y}_0).
\end{equation}
To compute the search directions effectively, the standard multilevel method considers the Galerkin model as a choice of the coarse model:
\begin{equation}
\begin{aligned}
F(\w{y})& := \langle \w{R}_k \nabla  f(\w{x}_k), \w{y} -\w{y}_0 \rangle  +
\frac12 \langle \w{R}_k\nabla^2 f(\w{x}_k)\w{P}_k (\w{y} -\w{y}_0),\w{y} -\w{y}_0 \rangle.
\end{aligned}
\label{eq:galerkinModel}
\end{equation}
It can be shown that combining \eqref{eq: coarse model} and \eqref{eq:galerkinModel} we obtain a closed form solution for the coarse direction:
\begin{equation*}
    \w{d}_k = - \w{P}_k \left( \w{R}_k\nabla^2 f(\w{x}_k)\w{P}_k \right)^{-1} \w{R}_k\nabla f(\w{x}_k)
\end{equation*}
However, it may be ineffective to employ solely the above strategy when computing the search direction if $\w{R}$ is fixed or defined in a deterministic fashion at each iteration. For instance, this can occur when $\nabla f(\w{x}) \in \operatorname{null} (\w{R})$ while $\| \nabla f(\w{x}_k) \|_2 \neq 0$, which implies no progress for the multilevel algorithm. The following conditions were introduced to prevent the multilevel algorithm from taking an ineffective coarse step:
\begin{equation*}
        \| \w{R} \nabla f (\w{x}) \| > \xi  \| \nabla f (\w{x}) \| \quad \text{and}
        \quad \| \w{R} \nabla f (\w{x}) \| > e,
\end{equation*}
where $\xi \in (0, \min (1, \| \w{R} \|))$ and $e>0$ are user defined parameters. Thus, the standard multilevel algorithm computes the coarse direction when the above conditions are satisfied and the fine direction (Newton) otherwise. On the other hand, it is expected that the multilevel algorithm will construct effective coarse directions when $\w{R}$ is selected randomly at each iteration. In particular, it has been demonstrated that the standard multilevel method can perform always coarse steps without compromising its superlinear or composite convergence rate (see for example \cite{ho2019newton, tsipinakis2021multilevel, tsipinakis2023multilevel}). 

\subsection{Simba}

The standard multilevel algorithm is well suited for strictly convex functions since the Hessian matrix is always positive definite. Due to the absence of positive-definiteness in general non-convex problems, Newton type methods may converge to a saddle point or a local maximum since they cannot guarantee the descent property of the search directions. They can even break down when the Hessian matrix is singular. These limitations have been efficiently tackled in the recent studies \cite{nesterov2006cubic}. In multilevel literature, SigmaSVD \cite{tsipinakis2023multilevel} addresses these limitations by using a truncated low-rank modification of the standard multilevel algorithm that efficiently escapes saddle points and flat areas. While SigmaSVD has been shown to perform well near saddles and flat areas, it requires $\mathcal{O}(m n_\ell^2)$ operations to form the reduced Hessian matrix which is prohibitively expensive for modern deep neural network models. Here, to alleviate this burden, we replace the expensive Hessian matrix with the outer product of the gradients. For this purpose, we consider the following coarse model:
\begin{equation}
\begin{aligned}
F(\w{y}) & := \langle \w{R}_k \w{G}_k,  \w{y} -\w{y}_0 \rangle +
\frac12 \langle (\w{R}_k \w{H}_k \w{P}_k)^{\frac{1}{2}} (\w{y} -\w{y}_0),\w{y} -\w{y}_0 \rangle,
\end{aligned}
\label{eq:coarseModelNew}
\end{equation}
where $\quad \w{H}_k := \w{G}_k \w{G}_k^T$ and  $\w{G}_k := \nabla f (\w{x}_k)$.
Thus, by defining $\hat{\w{d}}_\ell := \w{y}^* - \w{y}_0$ and $\w{Q}_{\ell, k} := \w{R}_k \w{H}_k \w{P}_k$, the coarse direction can be computed explicitly:
\begin{equation*}
\begin{aligned}
    \hat{\w{d}}_k & = \w{P}_k \left(  \underset{\w{d}_\ell \in\R^{n_\ell} }{\operatorname{arg \ min}}  \langle \w{R}_k \w{G}_k,  \w{d}_\ell \rangle +
\frac12 \langle \w{Q}_{\ell, k}^{\frac{1}{2}} \w{d}_\ell ,\w{d}_\ell \rangle \right) \\ 
& = - \w{P}_k \w{Q}_{\ell, k}^{-\frac{1}{2}}  \w{R}_k \w{G}_k.
\end{aligned}
\end{equation*}
To construct $\w{Q}_{\ell, k}^{-\frac{1}{2}}$, we perform a randomized truncated SVD to obtain the first $r+1$ eigenvalues and eigenvectors, that is, $\left[ \w{U}_{r+1}, \w{\Lambda}_{r+1} \right] = \text{T-SVD}(\w{Q}_{\ell,k})$ \cite{MR2806637}. Subsequently, we bound the diagonal matrix of eigenvalues from below by a scalar $m>0$:
\begin{equation} \label{eq: truncated eig matrix}
    \left[  \w{\Lambda}_{r+1} \right]_{i, m} := \begin{cases} \w{\Lambda}_{i}, & \w{\Lambda}_{i} \geq m \\
            m, & \text{otherwise}, 
            \end{cases}
\end{equation}
to obtain the new modified diagonal matrix as $\w{\Lambda}_{r+1}^m$. We then construct $\w{Q}_{\ell, k}^{-\frac{1}{2}}$ by treating all its eigenvalues below the $r$-th  equal to $r+1$ eigenvalue. Formally, we define
\begin{equation} \label{eq: Q_l inverse}
\begin{aligned}
    \w{Q}_{\ell,k}^{-\frac{1}{2}} & := \left[\w{\Lambda}_{r+1} \right]_{r+1, m}^{-\frac{1}{2}}  \w{I}_{n_\ell} +  \w{U}_r \left(  \left[\w{\Lambda}_{r+1}^m \right]^{-\frac{1}{2}} - \left[\w{\Lambda}_{r+1} \right]_{r+1, m}^{-\frac{1}{2}} \w{I}_r \right) \w{U}_r^T.
\end{aligned}
\end{equation}
The key component when forming the precondition matrix as above is that we do not allow for zero eigenvalues since zero eigenvalues indicate the presence of flat areas where the convergence rate of optimization methods significantly decays. Hence, by requiring $\w{Q}_{\ell,k} \succeq  m \w{I}_{n_\ell}$, we ensure that directions which correspond to flat curvatures will turn into directions whose curvature is positive, anticipating an accelerated convergence to the local minimum or a fast escape rate from saddles.  Similar techniques for constructing the preconditioners have been employed in \cite{tsipinakis2023multilevel, paternain2019newton} where it is demonstrated numerically that the algorithms can rapidly escape saddle points and flat areas in practical applications. The computational cost of forming the preconditioner matrix is $\mathcal{O}(n_\ell^2)$ which is significantly smaller than computing the Hessian matrix. Moreover, the randomized SVD requires $\mathcal{O}(r n_\ell^2)$ operation. 

In addition, the method can be trivially modified to employ the EMA of the gradients to account for the averaged history of the local information. Moreover, the EMA of the gradient is used to construct more informative preconditioner matrices. Specifically, to obtain the accelerated algorithm, we set a momentum parameter $\beta \in (0, 1)$ and replace $\w{G}_k$ in \eqref{eq:coarseModelNew} with the EMA update:
\begin{equation*}
   \w{G}_k := \beta \w{G}_{k-1} +  \nabla f (\w{x}_k),
\end{equation*}
where $\w{G}_0 = 0$ and $k \geq 1$. Simba with momentum is presented in Algorithm \ref{alg: alg}.

We emphasize that the method is scalable for solving large deep neural network models. We describe the three components that render the iterations of Simba efficient: \textbf{(a)} The operations can be performed in a decoupled manner. This means that preconditioners can be computed independently at each layer, leading into forming much smaller matrices that can be efficiently stored during the training. \textbf{(b)} The restriction operator is constructed based on uniform sampling without replacement from the rows of the identity matrix, $\w{I}_n$. This is described formally in the following definition: \begin{definition} \label{def: P}
\emph{ Let the set $S_n = \{1,2, \ldots, n\}$. Sample uniformly without replacement $n_\ell < n$ elements from $S_n$ to construct $S_{n_\ell} = \{s_{n_1}, s_{n_2}, \ldots, s_{n_\ell}\}$. Then, the $i^\text{th}$ row of $\w{R}$ is the $s_{n_i}^\text{th}$ row of the identity matrix $\w{I}_n$ and $\w{P}= \w{R}^T$}.    
\end{definition} 
This way of constructing the restriction operator yields efficient iterations since in practice the coarse model in \eqref{eq:coarseModelNew} can be formed by merely sampling $n_\ell$ elements from $\w{x}_k$ and  $\nabla f (\w{x}_k)$ which has negligible computational cost. \textbf{(c)} Taking advantage of the parameter structure. Methods that employ full matrix preconditioners further reduce the memory requirement by exploiting the matrix or tensor structure of the parameters (for details see Shampoo \cite{gupta2018shampoo}). For instance, for a matrix setting, if the parameter $\w{X} \in \R^{q \times d}$ and select $\w{R} \in \R^{n_\ell \times q}$, where $n_\ell < q$, then we obtain $\w{G}_{\ell, k} := \w{R} \w{G}_{k} \in \R^{n_\ell \times d}$ and $\w{Q}_{\ell,k} \in \R^{n_\ell \times n_\ell}$. This results in $\mathcal{O}((r+d) n_\ell^2)$ operations for forming the preconditioner and applying the randomized SVD. Note that this number is much smaller than $\mathcal{O}((d+q)(d^2 + q^2))$ of Shampoo.
\begin{figure*}[t]
\centering
\begin{subfigure}{.32\textwidth}
\centering
  \includegraphics[width=.98\linewidth]{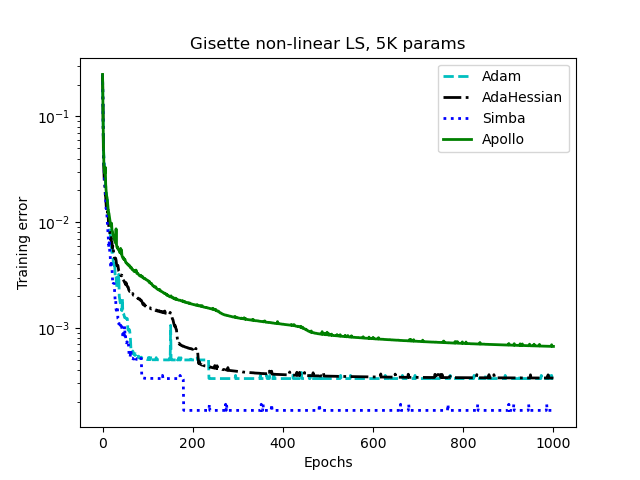}  
  \caption{Initialized at the origin}
  \label{subfig: nlls conv rate gisette}
\end{subfigure}
\begin{subfigure}{.32\textwidth}
\centering
  \includegraphics[width=.98\linewidth]{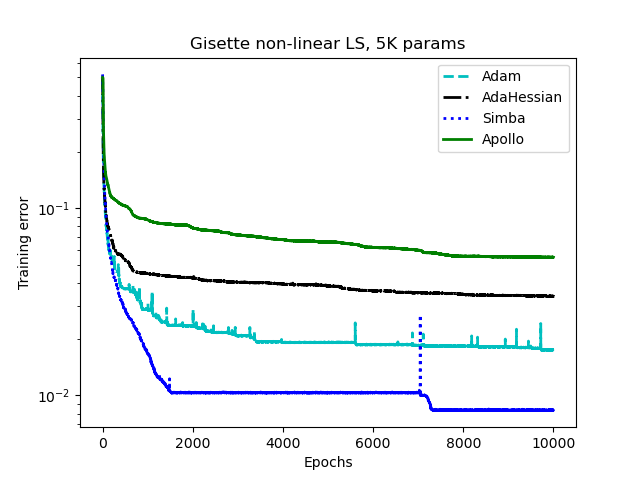}  
  \caption{Initialized randomly from $\mathcal{N}(0,1)$}
  \label{}
\end{subfigure}
\begin{subfigure}{.32\textwidth}
\centering
  \includegraphics[width=.98\linewidth]{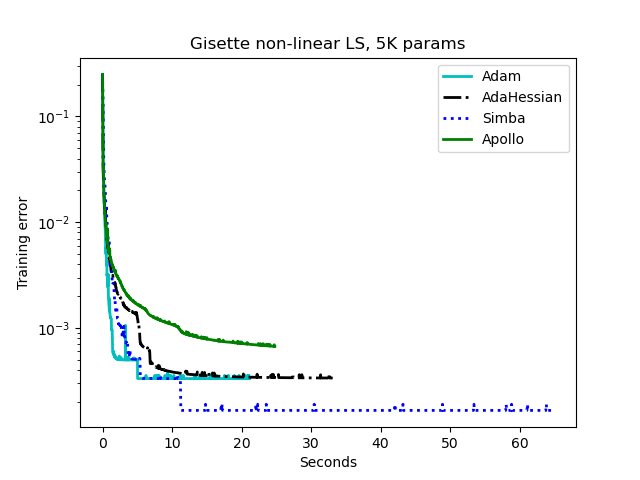}  
  \caption{Gissete with $\w{x}_0 = 0$}
  \label{subfig: gpu time nlls gisette}
\end{subfigure}
\caption{Convergence behaviour of various algorithms for the non-linear
least-squares problem. }
\label{fig: nlls algs}
\end{figure*}
\section{Convergence Analysis}

In this section we provide a simple convergence analysis of Simba
when it generates sequences using the coarse model in \eqref{eq:coarseModelNew} (without momentum). We show a linear convergence rate when the coarse model is constructed in both deterministic and randomized manner.
Our analysis is based on the classical theory that assumes strongly convex functions and Lipschitz continuous gradients. In deterministic scenarios, the method is expected to alternate between coarse and fine steps to always reduce the value function. For this reason, we present two convergence results: \textbf{(a)} when the coarse step is always accepted, and \textbf{(b)} when only fine steps are taken. Hence the complete convergence behaviour of Simba is provided. On the other hand, when the prolongation operator is defined randomly at each iteration, multilevel methods expected to converge using coarse steps only \cite{ho2019newton, tsipinakis2021multilevel, tsipinakis2023multilevel}. Our numerical experiments also verify this observation. Moreover, we derive the number of steps needed for the method to reach accuracy $\epsilon>0$ for both cases. We begin by stating our assumptions.

\begin{assumption}\label{ass: f}
There are scalars $0<\mu<L< +\infty$ such that $f$ is $\mu$-strongly convex and has $L$-Lipschitz continuous gradients if and only if:
\begin{enumerate}
    \item ($L$-Lipschitz continuity) for all $\w{x},\w{y} \in \R^n$
    \begin{equation*}
        \| \nabla f(\w{y}) - \nabla f(\w{x}) \|_2 \leq L \| \w{y} - \w{x}\|_2
    \end{equation*}
    \item ($\mu$-strong convexity) for all $\w{x},\w{y} \in \R^n$
    \begin{equation*}
        f(\w{y}) \geq f(\w{x}) + \langle \nabla f(\w{x}), \w{y} - \w{x} \rangle + \frac{\mu}{2} \|\w{y} - \w{x} \|_2
    \end{equation*}
\end{enumerate}
\end{assumption}

Next we state the assumptions on the prolongation and restriction operators.
\begin{assumption} \label{ass: P}
    For the restriction and prolongation operators $\w{R}$ and $\w{P}$ it holds that $\w{P} = \w{R}^T$ and $\operatorname{rank}(\w{P}) = n_\ell$.
\end{assumption}
The assumptions on the linear operators are not restrictive for practical application. For instance, Definition \ref{def: P} on $\w{P}$ and $\w{R}$ satisfies Assumption \ref{ass: P}. The following assumption ensures that the algorithm always selects effective coarse directions.

\begin{assumption} \label{ass: coarse}
    There exist $e>0$ and $\xi \in (0, \min (1, \| \w{R} \|_2))$ such that if $\| \nabla f (\w{x}) \|_2 \neq 0$ it holds
    \begin{equation*}
        \| \w{R} \nabla f (\w{x}) \|_2 > \xi  \| \nabla f (\w{x}) \|_2 \quad \text{and}
        \quad \| \w{R} \nabla f (\w{x}) \|_2 > e.
    \end{equation*}
\end{assumption}
For our convergence result below we will need the following quantity: we denote $\omega := \max \{\| \w{R}\|_2, \| \w{P}\|_2  \}$.

\begin{theorem} \label{theorem: coarse steps}
Let $f: \R^n \rightarrow \R$ be a function such that Assumption \ref{ass: f} holds. Suppose also that Assumptions \ref{ass: P} and \ref{ass: coarse} hold. Moreover, given $\w{Q}_{\ell,k}^{-\frac{1}{2}}$ in \eqref{eq: Q_l inverse}, define
    \begin{align*}
        \hat{\w{d}}_{k} & := - \w{P}_k \w{Q}_{\ell,k}^{-\frac{1}{2}} \w{R}_k \nabla f (\w{x}_k),
    \end{align*}
and suppose that the sequence $ (\w{x}_k)_{k \in \mathbb{N}}$ is generated by $\w{x}_{k+1} = x_{k} + \frac{\xi^2 m}{L \sqrt{M} \omega^4} \hat{\w{d}}_{k}$.
Then, there exists $\hat{c} \in (0,1)$ such that 
\begin{equation*}
    f(\w{x}_{k+1}) - f(\w{x}^*) \leq \hat{c} (f(\w{x}_{k}) - f(\w{x}^*)).
\end{equation*}
Moreover, at most
\begin{equation*}
    \hat{K} = \frac{\log\left( f(\w{x}_0) - f(\w{x}^*) / \epsilon  \right)}{\log(1/\hat{c})}
\end{equation*}
iterations are required for this process to reach accuracy $\epsilon$. 
\end{theorem}

\begin{proof}
Recall from the statement of the theorem that the sequence is generated by $\w{x}_{k+1} = \w{x}_{k} + t_k \hat{\w{d}}_k$, where $\hat{\w{d}}_k = - \w{P}_k \w{Q}_{\ell, k}^{-\frac{1}{2}}  \w{R}_k \nabla f(\w{x}_k)$, $t_k = \frac{\xi^2m}{\omega^4 L \sqrt{M}}$ and $\w{Q}_\ell^{-1/2}$ is defined in \eqref{eq: Q_l inverse}. We also define the following quantity
\begin{equation*}
    \hat{\lambda}({\w{x}}) :=  \sqrt{\nabla f(\w{x})^T \w{P} \w{Q}_{\ell}^{-\frac{1}{2}}  \w{R} \nabla f(\w{x})}.
\end{equation*}
It holds $\hat{\lambda}({\w{x}}) \geq 0$. Below, we collect two general results that will be useful later in the proof.
Given the Lipschitz continuity in Assumption \ref{ass: f} one can prove the following inequality (\cite{nesterov2018lectures}):
\begin{equation} \label{ineq: L-lipschitz}
    f(\w{y}) \leq f(\w{x}) + \nabla f(\w{x})^T (\w{y} - \w{x}) + \frac{L}{2} \| \w{y} - \w{x} \|_2
\end{equation}
Similarly, from the strong convexity we can obtain
\begin{equation*}
        f(\w{y}) \leq f(\w{x}) + \nabla f(\w{x})^T (\w{y} - \w{x}) + \frac{1}{2 \mu} \| \nabla f(\w{x}) - \nabla f(\w{y}) \|_2
\end{equation*}
Replacing $\w{y}$ and $\w{x}$ with $\w{x}_k$ and $\w{x}^*$, respectively, we have that
\begin{equation} \label{ineq: strong convexity}
        f(\w{x}_k) - f(\w{x}^*) \leq  \frac{1}{2 \mu} \| \nabla f(\w{x}_k) \|_2    
\end{equation}
Furthermore, by construction, $\w{Q}_{\ell,k}$ is bounded. It is bounded from below by $m$  from \eqref{eq: truncated eig matrix} and it is bounded from above by the Lipschitz continuity of the gradients. Then, there exists $M \geq m > 0$ such that for all $k \geq 1$ we have that
\begin{equation*}
    m \w{I} \preceq \w{Q}_{\ell,k} \preceq M \w{I}
\end{equation*}
The above inequality implies
\begin{equation*} \label{ineq: bounds on full preconditioner}
    \frac{1}{\sqrt{M}} \w{P}_k \w{R}_k \preceq \w{P}_k \w{Q}_{\ell,k}^{-\frac{1}{2}} \w{R}_k \preceq \frac{1}{\sqrt{m}} \w{P}_k \w{R}_k
\end{equation*}
Using the above inequality we can obtain the following bound
\begin{equation} \label{ineq: d_h upper bound}
\begin{aligned}
    \| \hat{\w{d}}_k \|_2 = \| \w{P}_k \w{Q}_{\ell, k}^{-\frac{1}{2}}  \w{R}_k \nabla f(\w{x}_k) \|_2 \leq \| \w{P}_k \w{Q}_{\ell, k}^{-\frac{1}{2}}  \w{R}_k \|_2 \|\nabla f(\w{x}_k)\|_2 \leq \frac{1}{\sqrt{m}} \| \w{P}_k  \w{R}_k \|_2 \|\nabla f(\w{x}_k)\|_2
    \leq \frac{\omega^2}{\sqrt{m}} \|\nabla f(\w{x}_k)\|_2.
\end{aligned}
\end{equation}
Similarly we can show a lower bound on $\hat{\lambda}({\w{x}})$
\begin{equation} \label{ineq: lambda lower bound}
        \hat{\lambda}({\w{x}})^2 =  \nabla f(\w{x})^T \w{P} \w{Q}_{\ell}^{-\frac{1}{2}}  \w{R} \nabla f(\w{x}) \geq \frac{1}{\sqrt{M}} \| \w{R} \nabla f(\w{x}) \|_2^2 \geq \frac{\xi^2}{\sqrt{M}} \| \nabla f(\w{x}_k) \|_2^2,
\end{equation}
where the last inequality follows from Assumption \ref{ass: coarse}.
Using now \eqref{ineq: L-lipschitz} and the fact that $\hat{\lambda}({\w{x}})^2 = - \nabla f(\w{x})^T \hat{\w{d}}_k$ we take
\begin{equation*}
    f(\w{x}_{k+1}) \leq f(\w{x}_{k}) - t_k \hat{\lambda}(\w{x}_k)^2 + t_k^2\frac{L}{2} \| \hat{\w{d}}_k \|^2_2.
\end{equation*}
Combining the above inequality with \eqref{ineq: lambda lower bound} and \eqref{ineq: d_h upper bound} we have that
\begin{equation*}
\begin{aligned}
    f(\w{x}_{k+1}) & \leq f(\w{x}_{k}) - t_k \frac{\xi^2}{\sqrt{M}} \| \nabla f(\w{x}) \|_2^2 + t_k^2\frac{L \omega^4}{2m} \| \nabla f(\w{x}_k) \|_2^2 \\
    & = f(\w{x}_{k}) - \frac{\xi^4 m}{2\omega^4 M L} \| \nabla f(\w{x}_k) \|_2^2,
\end{aligned}
\end{equation*}
where the last inequality follows by the definition of $t_k$. Adding and subtracting $f(\w{x}^*)$ on the above relationship and incorporating inequality \eqref{ineq: strong convexity} we obtain
\begin{equation*}
    f(\w{x}_{k+1}) - f(\w{x}^*) \leq \hat{c} (f(\w{x}_{k}) - f(\w{x}^*)),
\end{equation*}
where $\hat{c} := 1 - \frac{\xi^4 m \mu}{\omega^4 M L}$. Since $\xi < \omega, m < M$ and $\mu < L$ we take $\hat{c} \in (0, 1)$. Unravelling the last inequality we get
\begin{equation*}
    f(\w{x}_{k}) - f(\w{x}^*) \leq \hat{c}^k (f(\w{x}_{0}) - f(\w{x}^*)),
\end{equation*}
and thus $\lim_{k \rightarrow \infty} f(\w{x}_{k}) = f(\w{x}^*)$. Finally, solving for $k$ the inequality $\hat{c}^k (f(\w{x}_{0}) - f(\w{x}^*)) \leq \epsilon$, we conclude that at least
\begin{equation*}
    \hat{K} = \frac{\log\left( f(\w{x}_0) - f(\w{x}^*) / \epsilon  \right)}{\log(1/\hat{c})}
\end{equation*}
steps are required for this process to achieve accuracy $\epsilon$. 
\end{proof}

A direct consequence of the above theorem is convergence in expectation when $(\w{x}_k)_{k \geq 1}$ is generated randomly via a random prolongation matrix, e.g., see Definition \ref{def: P}. In this case we can guarantee that
\begin{equation*}
   \mathbb{E}[ f(\w{x}_{k})] - f(\w{x}^*) \leq \hat{c}^k (f(\w{x}_{0}) - f(\w{x}^*)),
\end{equation*}
which implies that $\lim_{k \to \infty} \mathbb{E}[ f(\w{x}_{k})] = f(\w{x}^*)$.
\begin{table}
\caption{Mean training error and standard deviation for the Non-Linear Least-Squares problem. The results were obtained over 5 runs for random initialization from $\mathcal{N}(0,1)$.}
\centering
\begin{tabular}{c|c}
\hline
Algorithm &  Mean $\pm$ Std\\
\hline
Adam & $0.0158 \pm 0.0016 $ \\
AdaHessian & $ 0.0325 \pm 0.0029 $ \\
Simba & \colorbox{BlueGreen}{$ 0.0097 \pm 0.001 $} \\
Apollo & $ 0.0556 \pm 0.0018 $ \\
\hline
\end{tabular}
\label{tab: nlls}
\end{table}
Theorem \ref{theorem: coarse steps} effectively shows the number of steps required for the method to reach the desired accuracy when the coarse direction is always effective. However, this may not be always true, which, in deterministic settings, implies no progress for the method. As discussed previously $\xi$ and $e$ can be viewed as user-defined parameters that prevent the method from taking the ineffective coarse steps.
Given fixed $\xi$ and $e$, the method performs iterations in the fine level if one of the two conditions in Assumption \ref{ass: coarse} is violated.
In the fine level, the method constructs the preconditioner as follows
\begin{equation*}
    \w{Q}_k := \nabla f (\w{x}_k) \nabla f (\w{x}_k)^T,
\end{equation*}
and then $\w{Q}_k^{-\frac{1}{2}}$ is constructed exactly as in $\w{Q}_{\ell, k}$ in \eqref{eq: Q_l inverse}.
The next theorem shows a linear rate and derives the number of steps required when the method takes only fine steps. 
\begin{theorem} \label{theorem: fine steps}
Let $f: \R^n \rightarrow \R$ be a function such that Assumption \ref{ass: f} holds. Suppose also that Assumption \ref{ass: P} holds. Moreover, define
    \begin{align*}
        \w{d}_{k} := - \w{Q}_k^{-\frac{1}{2}} \nabla f (\w{x}_k),
    \end{align*}
and suppose that the sequence $ (\w{x}_k)_{k \in \mathbb{N}}$ is generated by $\w{x}_{k+1} = x_{k} + \frac{m}{L \sqrt{M} } \hat{\w{d}}_{k}$.
Then, there exists $c \in (0,1)$ such that 
\begin{equation*}
    f(\w{x}_{k+1}) - f(\w{x}^*) \leq c (f(\w{x}_{k}) - f(\w{x}^*)).
\end{equation*}
Moreover, at most
\begin{equation*}
    K = \frac{\log\left( f(\w{x}_0) - f(\w{x}^*) / \epsilon  \right)}{\log(1/c)}
\end{equation*}
iterations are required for this process to reach accuracy $\epsilon$.
\end{theorem}

\begin{proof}
    The proof of theorem as it follows analogously to Theorem \ref{theorem: coarse steps}. The difference is the term that controls the linear rate which is now given by $c := 1 - \frac{m \mu}{M L}$. It holds $0 < \hat{c} \leq c < 1$.   
\end{proof}
As expected, theorem \ref{theorem: fine steps} shows a faster linear rate since the entire local information is employed during the training. Combining theorems \ref{theorem: coarse steps} and \ref{theorem: fine steps}, we provide the complete picture of the linear convergence rate of the proposed method.

\begin{figure*}
\centering
\begin{subfigure}{.24\textwidth}
\centering
  \includegraphics[width=.98\linewidth]{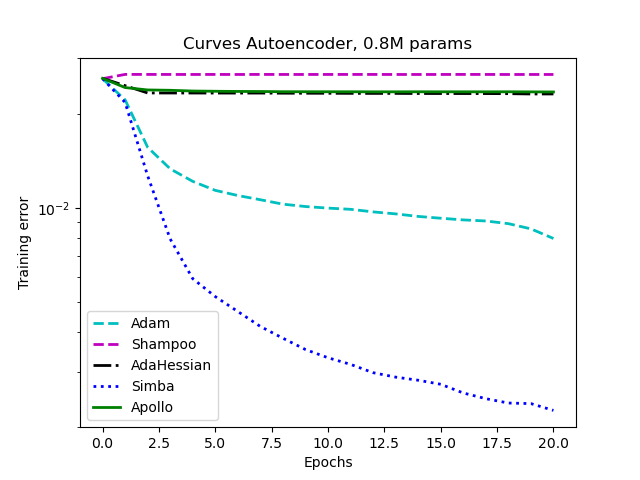}  
\end{subfigure}
\begin{subfigure}{.24\textwidth}
\centering
  \includegraphics[width=.98\linewidth]{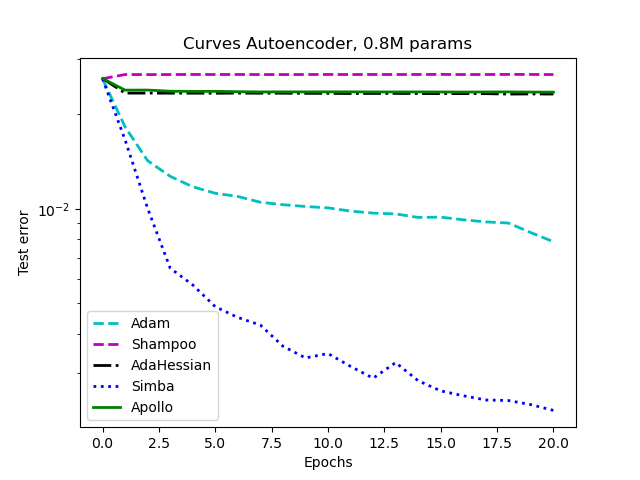}  
\end{subfigure}
\begin{subfigure}{.24\textwidth}
\centering
  \includegraphics[width=.98\linewidth]{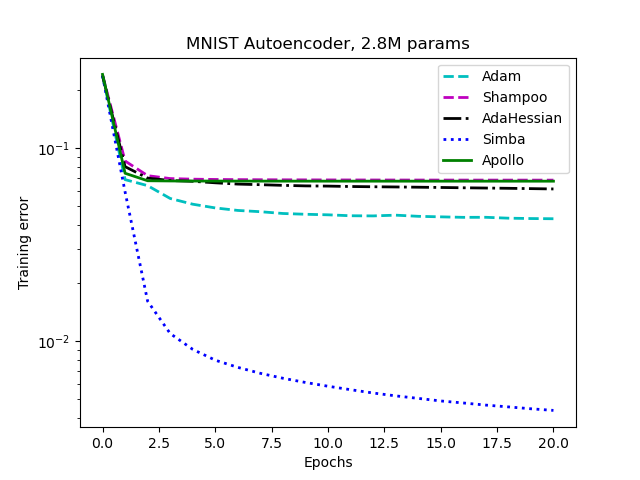}  
\end{subfigure}
\begin{subfigure}{.24\textwidth}
\centering
  \includegraphics[width=.98\linewidth]{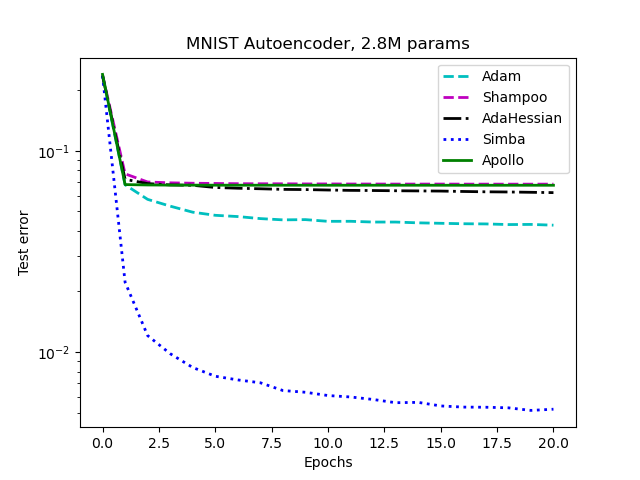}  
\end{subfigure}
\caption{Convergence behaviour of various algorithms for the CURVES and MNIST autoencoders. }
\label{fig: autoencoders}
\end{figure*}
\begin{table*}
\caption{Mean training error and standard deviation over $20$ and $5$ runs for the CURVES and MNIST autoencoders, respectively The last column reports the wall-clock time for a single run of the two autoencoders problem.}
\centering
\begin{tabular}{c|c|c|c}
\hline
Algorithm &  CURVES & MNIST & Seconds CURVES/MNIST\\
\hline
Adam & $0.0057 \pm 0.0017 $ & $0.0431 \pm 0.0007 $ & \colorbox{BlueGreen}{$22 / 176 $} \\
AdaHessian & $0.0165 \pm 0.0049 $& $0.0604  \pm 0.0044 $& $48 / 224 $ \\
Shampoo & $ - $& $-$& $606 / 12,442 $ \\
Apollo & $ 0.02309 \pm 0.00159 $& $0.0672 \pm 6.3 \times 10^{-7}$& $ 31 / 197 $ \\
Simba & \colorbox{BlueGreen}{$ 0.00250 \pm 0.0005 $}& \colorbox{BlueGreen}{$0.0043 \pm 8.1 \times 10^{-5} $}& $ 54 / 467 $ \\
\hline
\end{tabular}
\label{tab: autoencoders}
\end{table*}

\section{Numerical Experiments}
In this section we validate the efficiency of Simba to a number machine learning problems. Our goal is to illustrate that Simba outperforms the state-of-the-art diagonal optimization methods for problems with saddle points and flat areas or when the gradients are vanishing. For this purpose, we consider a non-linear least squares problems and two deep autoencoders where optimization methods often converge to suboptimal solutions. In addition, we demonstrate that our method is efficient and offers comparable, if not better, generalization errors compared to the state-of-the-art optimization methods on standard benchmark ResNets using CIFAR10 and CIFAR100 datasets.

\textbf{Algorithms and  set up:}
We compare Simba with momentum (algorithm \ref{alg: alg}) against Adam \cite{kingma2014adam}, AdaHessian \cite{yao2021adahessian}, Apollo \cite{ma2020apollo} and Shampoo \cite{gupta2018shampoo} on a Tesla T4 GPU with a 16G RAM. The GPU time for 
Shampoo is not comparable to that of the other algorithm 
and hence its behaviour is reported only for the autoencoder problems.
For all algorithms, the learning rate was selected using 
a grid search for $t_k \in $\{1e-4, 5e-4, 1e-3, 5e-3, 
1e-2, 5e-2, 1e-1, 5e-1, 1\}. For all the algorithms we 
set the momentum parameters to their default 
values. For Simba we set $r=20$ in all experiments.
Through all the experiments the batch size is set to 
$128$. For all algorithms we comprehensively tuned 
$\operatorname{eps}$; the default value is selected when others do not
yield an improvement. In our case we denote $m \equiv \operatorname{eps}$.

\subsection{Non-linear least-squares}

Given a training dataset $\{\w{a}_i, b_i\}_{i=1}^m$, 
$\w{a}_i \in \mathbb{R}^n$ and $b_i \in \mathbb{R}$, we 
consider solving the following non-linear least-squares 
problem
\begin{equation*}
	\min_{\w{x} \in \R^n} \frac{1}{m} \sum_{i=1}^{m}
	\left(b_i - g(\w{a}_i^T \w{x}) \right)^2, \quad g(\omega) := 
	\frac{1}{1 + \exp(\omega)}, 
\end{equation*}
which is a non-convex optimization
problem. Here, we consider the Gisette dataset \footnote{Dataset available at: \url{https://www.csie.ntu.edu.tw/~cjlin/libsvmtools/datasets/binary.html}}
for which $m = 6000$ and $n = 5000$. Furthermore, for Adam we select $t_k = 0.001$ and 
$\operatorname{eps} = 10^{-8}$ while for Apollo we set 
$t_k = 0.01$ and $\operatorname{eps}=10^{-4}$. For 
AdaHessian, $t_k = 0.1$ and the hessian power parameter is set to 
to $0.5$. For Simba we select $n_\ell = 250, t_k = 0.05$ and 
$\operatorname{eps} = 10^{-12}$. The performance of optimization methods appear in Figure \ref{fig: nlls algs}.
Observe that this problems has several flat areas or 
saddle points which slow down the convergence of all algorithms. Nevertheless, the method with the best behaviour is Simba which enjoys a much faster escape rate from the saddle points and thus always returns
lower training errors. This can be also observed 
for different initialization points. In Table \ref{tab: nlls} we report the average training error and the standard deviation over $5$ runs. The method that comes closer to Simba is Adam.
Further, from Figure \ref{fig: nlls algs} we see that, although the wall-clock time of Simba is two to three times increased
compared to its competitors, the total time of our method is much better due to its fast escape rate near saddles and flat areas. Further, in Figure \ref{fig: nlls Ns} in the supplementary material we compare the performance of Simba for different sizes in the coarse model. We observed that even Simba with $n_\ell = 25$ (i.e., updating only $25$ parameters at each iteration) enjoys better escape rate than its competitors. This indicates that diagonal methods perform poorly near saddles and flat areas and highlights the importance of constructing meaningful preconditioners. Moreover, observe that the best wall-clock time is achieved with only $5\%$ of the dimensions. This indicates that Simba significantly reduces the computational cost of preconditioned methods without compromising the convergence rate.

\begin{figure*}
\centering
\begin{subfigure}{.24\textwidth}
\centering
  \includegraphics[width=.98\linewidth]{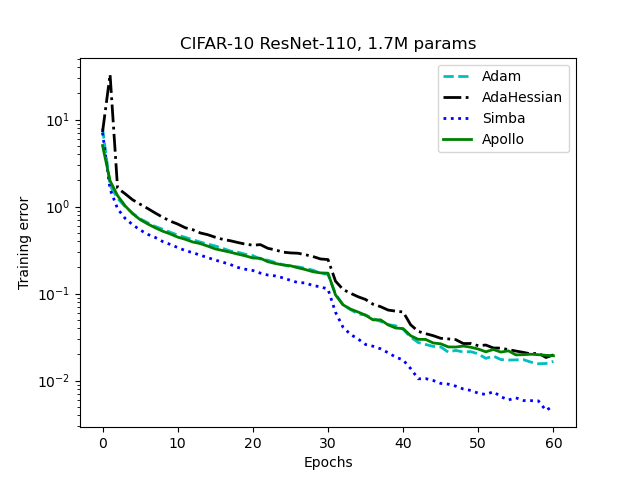}  
\end{subfigure}
\begin{subfigure}{.24\textwidth}
\centering
  \includegraphics[width=.98\linewidth]{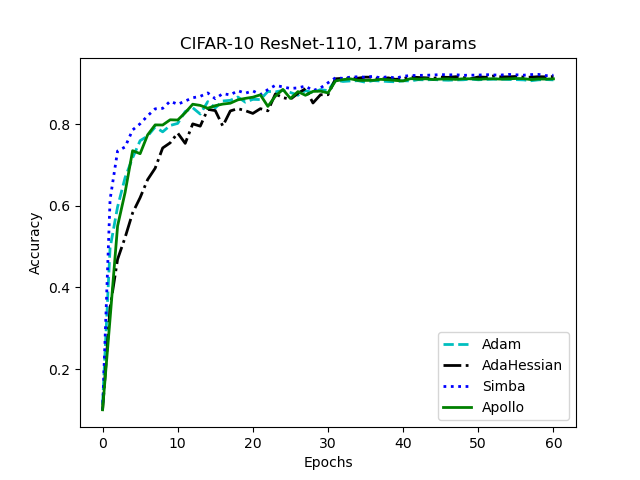}  
\end{subfigure}
\begin{subfigure}{.24\textwidth}
\centering
  \includegraphics[width=.98\linewidth]{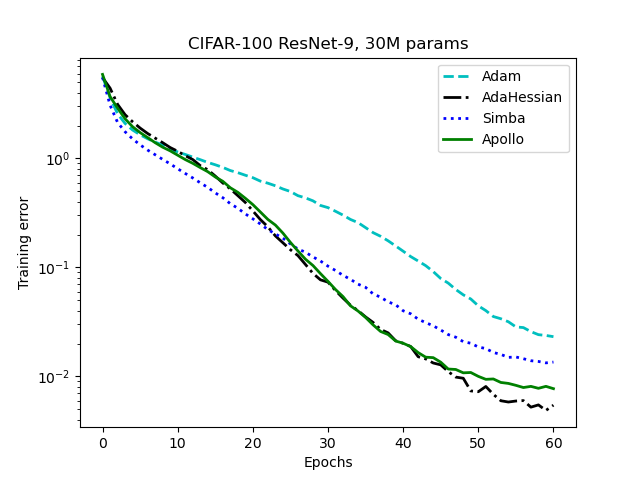}  
\end{subfigure}
\begin{subfigure}{.24\textwidth}
\centering
  \includegraphics[width=.98\linewidth]{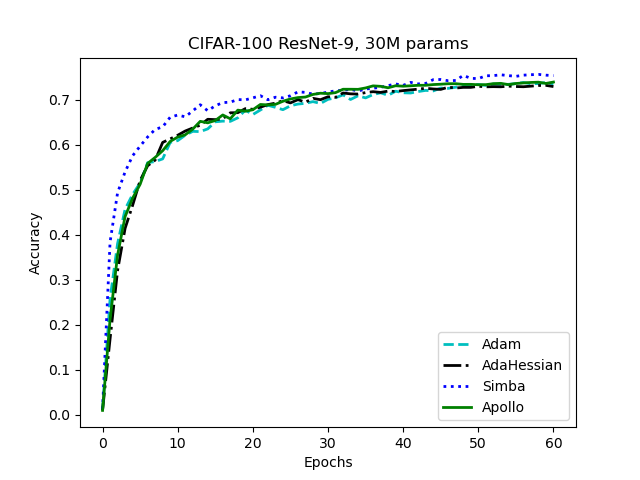}  
\end{subfigure}
\caption{Convergence behaviour of various algorithms on CIFAR10 
using ResNet110 and CIFAR100 using ResNet9.}
\label{fig: resnet110 and resnet9}
\end{figure*}
\begin{table*}
\caption{Mean training error and accuracy $\pm$ standard deviation of various optimization algorithms over $5$ on CIFAR10 and CIFAR100, respectively.}
\centering
\begin{tabular}{c|c|c|c|c}
\hline
  &\multicolumn{2}{ c| }{ ResNet-110 - CIFAR10} & \multicolumn{2}{ c }{ResNet-9 - CIFAR100} \\
\hline
Algorithm &  Training Error  & Accuracy & Training Error & Accuracy\\
\hline
Adam & $0.0284 \pm 0.0060 $ & $91.03 \pm 0.16 $ & $0.0237 \pm 0.0005 $ & $73.91 \pm 0.18$ \\

AdaHessian & $0.0560 \pm 0.0427 $& $91.34  \pm 1.04 $& $0.0081 \pm 0.0036 $ &
$72.93 \pm 0.66$\\

Apollo & $ 0.0315 \pm 0.0066 $& $91.18 \pm 0.25$& \colorbox{BlueGreen}{ $ 0.0076 \pm 0.0003 $} & $73.70 \pm 0.19$\\

Simba & \colorbox{BlueGreen}{$ 0.0097 \pm 0.0023 $}& \colorbox{BlueGreen}{$92.07 \pm 0.09$}& $ 0.0133 \pm 0.0001 $ & \colorbox{BlueGreen}{$75.41 \pm 0.22$}\\
\hline
\end{tabular}
\label{tab: resnets}
\end{table*}

\subsection{Deep autoencoders}
In this section we investigate the performance of Simba on two deep autoencoder optimization problems that arise from (\cite{hinton2006reducing}), named CURVES and MNIST. Both optimization problems are considered difficult to solve and have become a standard benchmark for optimization methods due to the presence of saddle points and the issue of the vanishing gradients.
The CURVES\footnote{Dataset available at: \url{www.cs.toronto.edu/~jmartens/
digs3pts_1.mat}} autoencoder consists of an encoder with layer size $28\times28, 400, 200,100,50,25,6$, a symmetric decoder which totals to $0.8$M parameters.
For all layers the $\operatorname{relu}$ activation function is applied. The network was trained on $20,000$ images and tested on $10,000$. 
We set the learning rate equal to $0.001, 1, 0.05, 0.5$ and $0.01$ for Adam, Apollo, AdaHessian, Shampoo and Simba, respectively. 
The $\operatorname{eps}$ parameter is selected $10^{-15}$ and $10^{-8}$ for Apollo and Simba, respectively, and set to its default value for the rest algorithms. The hessian power parameter for AdaHessian is selected $1$. 
Further, the MNIST\footnote{Dataset available at: \url{https://pytorch.org/vision/main/datasets.html}} autoencoder consists of an encoder with layer size $28\times28, 1000, 500, 250, 30$ and a symmetric decoder which totals to $2.8$M parameters. In this network we use the $\operatorname{sigmoid}$ activation function to all layers. 
The training set for the MNIST dataset consists of $60,000$ images while the test set has $10,000$ images. 
Here, $t_k$ is set equal to $0.001, 0.9, 0.1, 0.05$ and $0.05$ for Adam, Apollo, AdaHessian, Shampoo and Simba, respectively. 
The $\operatorname{eps}$ parameter is set equal to $10^{-8}$ for Simba and to its default value for all other algorithms; the hessian power of AdaHessian is set equal to $1$. For both autoencoder problems the coarse model size parameter is set $n_\ell = 0.5N$.

The comparison between the optimization algorithms for the two deep autoencoders appears in Table  \ref{tab: autoencoders} and Figure \ref{fig: autoencoders}. Clearly, Simba performs the best among the optimization algorithms resulting in the lowest training error in both CURVES and MNIST autoencoders. As a result, Simba performs better than its competitors on the test set too. It is evident that all the other methods get stuck in stationary points with large errors. In these points we observed that the gradients become almost zero which results in a very slow progress of Adam, AdaHessian, Apollo and Shampoo. On the other, Simba enjoys a very fast escape rate from such points and thus it yields a very fast decrease in the training error.  In addition, Table  \ref{tab: autoencoders} reports the wall-clock time for each optimization algorithm. We see that Simba has an increased per-iteration costs (by a factor between two and three compared to Adam), nevertheless this is a reasonable price to pay for achieving a good decrease in the value function.
The averaged behaviour of Shampoo is missing due to its expensive iterations. For this reason we omit Shampoo from the following experiments.

\subsection{Residual Neural Networks}
In this section we report comparisons on the convergence and generalization performance between the optimization algorithms using ResNet110\footnote{Network implementation available at: \url{https://github.com/akamaster/pytorch_resnet_cifar10}} and ResNet9\footnote{Network implementation available at: \url{https://www.kaggle.com/code/yiweiwangau/cifar-100-resnet-pytorch-75-17-accuracy}} with CIFAR-10 and CIFAR-100 datasets, respectively. 
ResNet110 consists of 110 layers and a total of $1.7$M parameters while ResNet9 has 9 layers and $30$M parameters. Hence, our goal is to investigate the performance of Simba on benchmark problems and different architectures, i.e., deep and wide networks.

For ResNet110 we set an initial learning rate parameter equal to $0.001, 0.1, 0.1, 0.01$ for Adam, Apollo, AdaHessian and Simba, respectively. We train the network for a total number of $60$ epochs and decrease the learning rate by a factor of $5$ at epochs $30$ and $40$. The hessian power parameter for AdaHessian is selected $1$ and for Simba the hyper-parameter $\operatorname{eps}$ is set to $10^{-8}$ while $n_\ell = 0.5N$. The convergence and generalization performance between the optimization algorithms is illustrated in Figure \ref{fig: resnet110 and resnet9} and Table \ref{tab: resnets}. We see that Simba is able to achieve  results that are comparable, if not better, with that of Adam, AdaHessian and Apollo in both training and generalization errors. The total GPU time of Simba is about three times larger than that of Adam, which is the fastest algorithm, but it is considerably smaller than that of AdaHessian (see Figure \ref{fig: resnet110} in the supplement).

For ResNet9, we set an initial value on $t_k$ as follows:
$0.0005,  0.001,  0.05, 0.005$ for Adam, Apollo, AdaHessian and Simba, respectively. We train the network for $60$ epochs and use cosine annealing to determine the learning rate at each epoch and set minimum $t_k$ value equal to $0.01t_k, 0.1t_k, 0.01t_k, 0.02t_k$ for Adam, Apollo, AdaHessian and Simba, respectively. The hessian power, $\operatorname{eps}$ and $n_\ell$ parameters are set as above. 
Since classifying CIFAR100 is a much more difficult problem than CIFAR10, to improve the accuracy we use weight decay and gradient clipping with values fixed at $0.001$ and $0.01$, respectively, for all algorithms. Figure \ref{fig: resnet110 and resnet9} and Table \ref{tab: resnets} indicate that Simba is able to offer very good generalization results on a difficult task. Further, we see that Simba is less than two times slower in wall-clock time than Adam and Apollo but again considerably faster than AdaHessian (see Figure \ref{fig: resnet9} in the supplement). Figure \ref{fig: resnet9} also shows that the total GPU time of Simba for reaching the desired classification accuracy is comparable to that of Adam and Apollo which indicates the efficiency of our method on large-scale optimization problems.

\section{Conclusions}
We present Simba, a scalable bilevel optimization algorithm to address the limitations of first-order methods in non-convex settings. We demonstrate the fast escape rate from saddles of our method through empirical evidence. Numerical results also indicate that Simba achieves good generalization errors on modern machine learning applications. Convergence guarantees of Simba when locally assume convex functions are also established. As a future work, we aim to apply Simba for the training of LLMs such as GPT.

\bibliographystyle{unsrtnat}
\bibliography{references}  

\begin{thebibliography}{31}
\providecommand{\natexlab}[1]{#1}
\providecommand{\url}[1]{\texttt{#1}}
\expandafter\ifx\csname urlstyle\endcsname\relax
  \providecommand{\doi}[1]{doi: #1}\else
  \providecommand{\doi}{doi: \begingroup \urlstyle{rm}\Url}\fi

\bibitem[Dauphin et~al.(2014)Dauphin, Pascanu, Gulcehre, Cho, Ganguli, and
  Bengio]{dauphin2014identifying}
Yann~N Dauphin, Razvan Pascanu, Caglar Gulcehre, Kyunghyun Cho, Surya Ganguli,
  and Yoshua Bengio.
\newblock Identifying and attacking the saddle point problem in
  high-dimensional non-convex optimization.
\newblock \emph{Advances in neural information processing systems}, 27, 2014.

\bibitem[Panageas et~al.(2019)Panageas, Piliouras, and Wang]{panageas2019first}
Ioannis Panageas, Georgios Piliouras, and Xiao Wang.
\newblock First-order methods almost always avoid saddle points: The case of
  vanishing step-sizes.
\newblock \emph{Advances in Neural Information Processing Systems}, 32, 2019.

\bibitem[Nesterov et~al.(2018)]{nesterov2018lectures}
Yurii Nesterov et~al.
\newblock \emph{Lectures on convex optimization}, volume 137.
\newblock Springer, 2018.

\bibitem[Boyd and Vandenberghe(2004)]{MR2061575}
Stephen Boyd and Lieven Vandenberghe.
\newblock \emph{Convex optimization}.
\newblock Cambridge University Press, Cambridge, 2004.
\newblock ISBN 0-521-83378-7.
\newblock \doi{10.1017/CBO9780511804441}.
\newblock URL \url{https://doi.org/10.1017/CBO9780511804441}.

\bibitem[Nesterov(2004)]{MR2142598}
Yurii Nesterov.
\newblock \emph{Introductory lectures on convex optimization}, volume~87 of
  \emph{Applied Optimization}.
\newblock Kluwer Academic Publishers, Boston, MA, 2004.
\newblock ISBN 1-4020-7553-7.
\newblock \doi{10.1007/978-1-4419-8853-9}.
\newblock URL \url{https://doi.org/10.1007/978-1-4419-8853-9}.
\newblock A basic course.

\bibitem[Duchi et~al.(2011)Duchi, Hazan, and Singer]{duchi2011adaptive}
John Duchi, Elad Hazan, and Yoram Singer.
\newblock Adaptive subgradient methods for online learning and stochastic
  optimization.
\newblock \emph{Journal of machine learning research}, 12\penalty0 (7), 2011.

\bibitem[Tieleman et~al.(2012)Tieleman, Hinton, et~al.]{tieleman2012lecture}
Tijmen Tieleman, Geoffrey Hinton, et~al.
\newblock Lecture 6.5-rmsprop: Divide the gradient by a running average of its
  recent magnitude.
\newblock \emph{COURSERA: Neural networks for machine learning}, 4\penalty0
  (2):\penalty0 26--31, 2012.

\bibitem[Zeiler(2012)]{zeiler2012adadelta}
Matthew~D Zeiler.
\newblock Adadelta: an adaptive learning rate method.
\newblock \emph{arXiv preprint arXiv:1212.5701}, 2012.

\bibitem[Kingma and Ba(2014)]{kingma2014adam}
Diederik~P Kingma and Jimmy Ba.
\newblock Adam: A method for stochastic optimization.
\newblock \emph{arXiv preprint arXiv:1412.6980}, 2014.

\bibitem[Zaheer et~al.(2018)Zaheer, Reddi, Sachan, Kale, and
  Kumar]{zaheer2018adaptive}
Manzil Zaheer, Sashank Reddi, Devendra Sachan, Satyen Kale, and Sanjiv Kumar.
\newblock Adaptive methods for nonconvex optimization.
\newblock \emph{Advances in neural information processing systems}, 31, 2018.

\bibitem[Yao et~al.(2021)Yao, Gholami, Shen, Mustafa, Keutzer, and
  Mahoney]{yao2021adahessian}
Zhewei Yao, Amir Gholami, Sheng Shen, Mustafa Mustafa, Kurt Keutzer, and
  Michael Mahoney.
\newblock Adahessian: An adaptive second order optimizer for machine learning.
\newblock In \emph{proceedings of the AAAI conference on artificial
  intelligence}, volume~35, pages 10665--10673, 2021.

\bibitem[Ma(2020)]{ma2020apollo}
Xuezhe Ma.
\newblock Apollo: An adaptive parameter-wise diagonal quasi-newton method for
  nonconvex stochastic optimization.
\newblock \emph{arXiv preprint arXiv:2009.13586}, 2020.

\bibitem[Jahani et~al.(2021)Jahani, Rusakov, Shi, Richt{\'a}rik, Mahoney, and
  Tak{\'a}{\v{c}}]{jahani2021doubly}
Majid Jahani, Sergey Rusakov, Zheng Shi, Peter Richt{\'a}rik, Michael~W
  Mahoney, and Martin Tak{\'a}{\v{c}}.
\newblock Doubly adaptive scaled algorithm for machine learning using
  second-order information.
\newblock \emph{arXiv preprint arXiv:2109.05198}, 2021.

\bibitem[Liu et~al.(2023)Liu, Li, Hall, Liang, and Ma]{liu2023sophia}
Hong Liu, Zhiyuan Li, David Hall, Percy Liang, and Tengyu Ma.
\newblock Sophia: A scalable stochastic second-order optimizer for language
  model pre-training.
\newblock \emph{arXiv preprint arXiv:2305.14342}, 2023.

\bibitem[Broyden(1967)]{broyden1967quasi}
Charles~G Broyden.
\newblock Quasi-newton methods and their application to function minimisation.
\newblock \emph{Mathematics of Computation}, 21\penalty0 (99):\penalty0
  368--381, 1967.

\bibitem[Erdogdu and Montanari(2015)]{erdogdu2015convergence}
Murat~A Erdogdu and Andrea Montanari.
\newblock Convergence rates of sub-sampled newton methods.
\newblock In \emph{Proceedings of the 28th International Conference on Neural
  Information Processing Systems-Volume 2}, pages 3052--3060. MIT Press, 2015.

\bibitem[Pilanci and Wainwright(2017)]{pilanci2017newton}
Mert Pilanci and Martin~J Wainwright.
\newblock Newton sketch: A near linear-time optimization algorithm with
  linear-quadratic convergence.
\newblock \emph{SIAM Journal on Optimization}, 27\penalty0 (1):\penalty0
  205--245, 2017.

\bibitem[Xu et~al.(2016)Xu, Yang, Roosta-Khorasani, R{\'e}, and
  Mahoney]{xu2016sub}
Peng Xu, Jiyan Yang, Farbod Roosta-Khorasani, Christopher R{\'e}, and Michael~W
  Mahoney.
\newblock Sub-sampled newton methods with non-uniform sampling.
\newblock In \emph{Advances in Neural Information Processing Systems}, pages
  3000--3008, 2016.

\bibitem[Xu et~al.(2020)Xu, Roosta, and Mahoney]{xu2020second}
Peng Xu, Fred Roosta, and Michael~W Mahoney.
\newblock Second-order optimization for non-convex machine learning: An
  empirical study.
\newblock In \emph{Proceedings of the 2020 SIAM International Conference on
  Data Mining}, pages 199--207. SIAM, 2020.

\bibitem[Gupta et~al.(2018)Gupta, Koren, and Singer]{gupta2018shampoo}
Vineet Gupta, Tomer Koren, and Yoram Singer.
\newblock Shampoo: Preconditioned stochastic tensor optimization.
\newblock In \emph{International Conference on Machine Learning}, pages
  1842--1850. PMLR, 2018.

\bibitem[Tsipinakis and Parpas(2021)]{tsipinakis2021multilevel}
Nick Tsipinakis and Panos Parpas.
\newblock A multilevel method for self-concordant minimization.
\newblock \emph{arXiv preprint arXiv:2106.13690}, 2021.

\bibitem[Tsipinakis et~al.(2023)Tsipinakis, Tigkas, and
  Parpas]{tsipinakis2023multilevel}
Nick Tsipinakis, Panagiotis Tigkas, and Panos Parpas.
\newblock A multilevel low-rank newton method with super-linear convergence
  rate and its application to non-convex problems.
\newblock \emph{arXiv preprint arXiv:2305.08742}, 2023.

\bibitem[Halko et~al.(2011)Halko, Martinsson, and Tropp]{MR2806637}
N.~Halko, P.~G. Martinsson, and J.~A. Tropp.
\newblock Finding structure with randomness: probabilistic algorithms for
  constructing approximate matrix decompositions.
\newblock \emph{SIAM Rev.}, 53\penalty0 (2):\penalty0 217--288, 2011.
\newblock ISSN 0036-1445.
\newblock \doi{10.1137/090771806}.
\newblock URL \url{https://doi.org/10.1137/090771806}.

\bibitem[Bubeck et~al.(2015)]{bubeck2015convex}
S{\'e}bastien Bubeck et~al.
\newblock Convex optimization: Algorithms and complexity.
\newblock \emph{Foundations and Trends{\textregistered} in Machine Learning},
  8\penalty0 (3-4):\penalty0 231--357, 2015.

\bibitem[Reddi et~al.(2018)Reddi, Zaheer, Sra, Poczos, Bach, Salakhutdinov, and
  Smola]{reddi2018generic}
Sashank Reddi, Manzil Zaheer, Suvrit Sra, Barnabas Poczos, Francis Bach, Ruslan
  Salakhutdinov, and Alex Smola.
\newblock A generic approach for escaping saddle points.
\newblock In \emph{International conference on artificial intelligence and
  statistics}, pages 1233--1242. PMLR, 2018.

\bibitem[O'Leary-Roseberry et~al.(2020)O'Leary-Roseberry, Alger, and
  Ghattas]{o2020low}
Thomas O'Leary-Roseberry, Nick Alger, and Omar Ghattas.
\newblock Low rank saddle free newton: A scalable method for stochastic
  nonconvex optimization.
\newblock \emph{arXiv preprint arXiv:2002.02881}, 2020.

\bibitem[Martens and Grosse(2015)]{martens2015optimizing}
James Martens and Roger Grosse.
\newblock Optimizing neural networks with kronecker-factored approximate
  curvature.
\newblock In \emph{International conference on machine learning}, pages
  2408--2417. PMLR, 2015.

\bibitem[Ho et~al.(2019)Ho, Ko{\v{c}}vara, and Parpas]{ho2019newton}
Chin~Pang Ho, Michal Ko{\v{c}}vara, and Panos Parpas.
\newblock Newton-type multilevel optimization method.
\newblock \emph{Optimization Methods and Software}, pages 1--34, 2019.

\bibitem[Nesterov and Polyak(2006)]{nesterov2006cubic}
Yurii Nesterov and Boris~T Polyak.
\newblock Cubic regularization of newton method and its global performance.
\newblock \emph{Mathematical Programming}, 108\penalty0 (1):\penalty0 177--205,
  2006.

\bibitem[Paternain et~al.(2019)Paternain, Mokhtari, and
  Ribeiro]{paternain2019newton}
Santiago Paternain, Aryan Mokhtari, and Alejandro Ribeiro.
\newblock A newton-based method for nonconvex optimization with fast evasion of
  saddle points.
\newblock \emph{SIAM Journal on Optimization}, 29\penalty0 (1):\penalty0
  343--368, 2019.

\bibitem[Hinton and Salakhutdinov(2006)]{hinton2006reducing}
Geoffrey~E Hinton and Ruslan~R Salakhutdinov.
\newblock Reducing the dimensionality of data with neural networks.
\newblock \emph{science}, 313\penalty0 (5786):\penalty0 504--507, 2006.

\end{thebibliography}

\newpage
\appendix

\section{Extra Numerical Results}
Here, we report some extra numerical results. Figure \ref{fig: nlls Ns} shows how the size of the coarse model affects the convergence behaviour of Simba on the non-linear least squares problem. Additionally, Figures \ref{fig: resnet110} and \ref{fig: resnet9} show the wall-clock time for training the residual neural networks for 60 epochs.

\begin{figure}
\centering
\begin{subfigure}{.48\textwidth}
\centering
  \includegraphics[width=.98\linewidth]{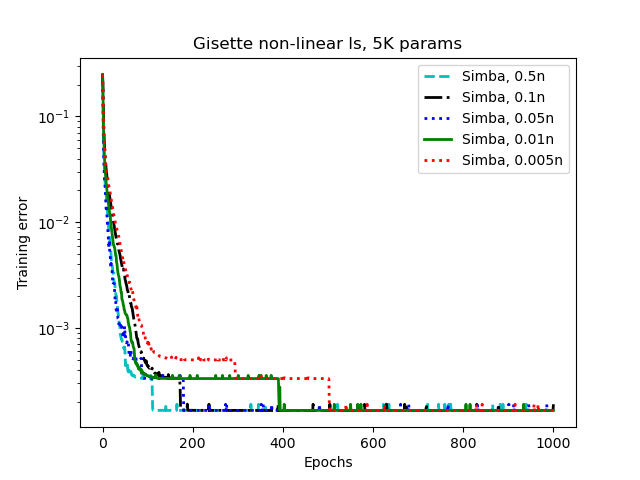}  
\end{subfigure}
\begin{subfigure}{.48\textwidth}
\centering
  \includegraphics[width=.98\linewidth]{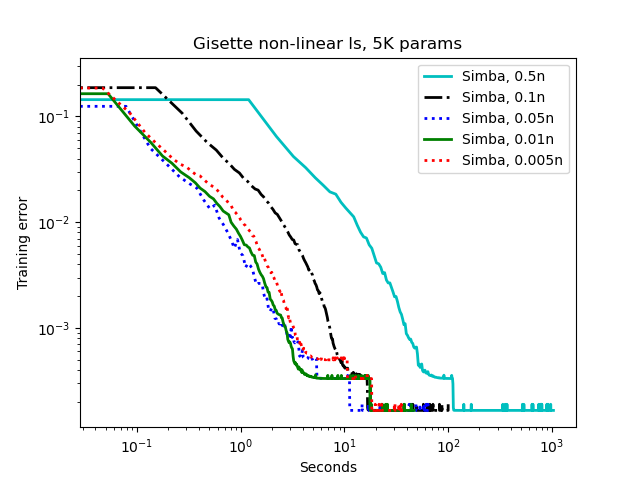}  
\end{subfigure}
\caption{Convergence behaviour of Simba for different coarse models. In all cases Simba is initialized at the origin. The parameter $r$ is set to $10$ for $n_\ell = 0.005N$ and $20$ otherwise. Further, $t_k$ is selected $0.01$ for $n=0.5N$ or $n=0.1N$ and $0.05$ otherwise. Interestingly,  we see that Simba with $n = 25$ matches the performance of Simba that use much larger coarse models. Note that, by Definition \ref{def: P}, Simba only updates $25$ parameters at each iteration which indicates that near saddles and flat areas one should employ more informative directions than those of the first-order methods with second-order momentum.  As expected, we see that a better convergence and escape rate is expected for larger coarse models but this comes at the cost of more expensive iterations. The figure shows that Simba with $n=250$ has the best total complexity.}
\label{fig: nlls Ns}
\end{figure}

\begin{figure}
\centering
\begin{subfigure}{.24\textwidth}
\centering
  \includegraphics[width=.98\linewidth]{res/Train_Loss_ResNet110.png}  
\end{subfigure}
\begin{subfigure}{.24\textwidth}
\centering
  \includegraphics[width=.98\linewidth]{res/Accuracy_ResNet110.png}  
\end{subfigure}
\begin{subfigure}{.24\textwidth}
\centering
  \includegraphics[width=.98\linewidth]{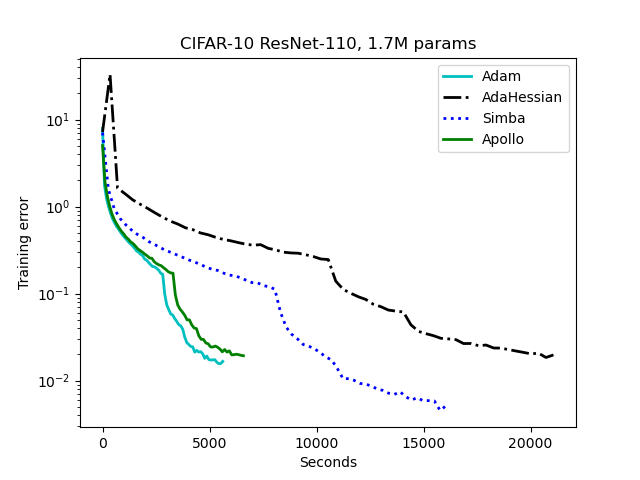}  
\end{subfigure}
\begin{subfigure}{.24\textwidth}
\centering
  \includegraphics[width=.98\linewidth]{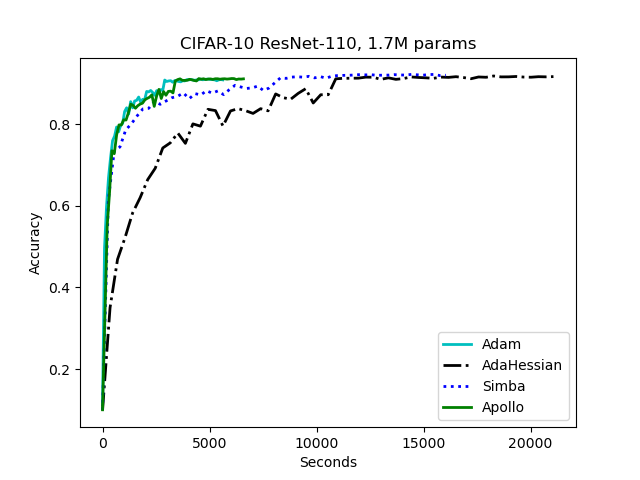}  
\end{subfigure}
\caption{Convergence behaviour and wall-clock time of various algorithms on CIFAR10 
using ResNet110.}
\label{fig: resnet110}
\end{figure}

\begin{figure}
\centering
\begin{subfigure}{.24\textwidth}
\centering
  \includegraphics[width=.98\linewidth]{res/Train_Loss_ResNet9.png}  
\end{subfigure}
\begin{subfigure}{.24\textwidth}
\centering
  \includegraphics[width=.98\linewidth]{res/AccuracyResNet9.png}  
\end{subfigure}
\begin{subfigure}{.24\textwidth}
\centering
  \includegraphics[width=.98\linewidth]{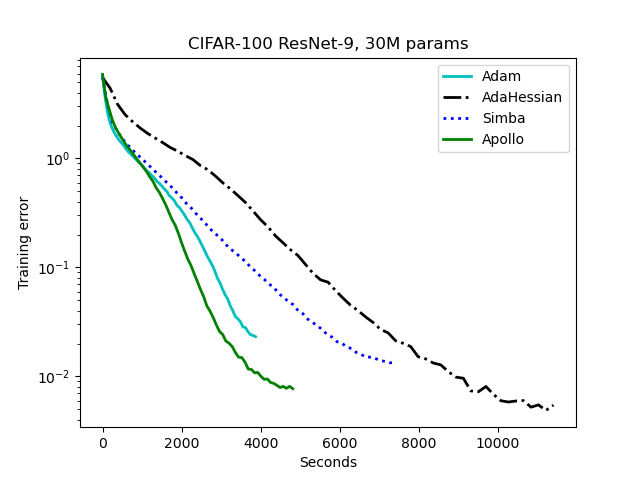}  
\end{subfigure}
\begin{subfigure}{.24\textwidth}
\centering
  \includegraphics[width=.98\linewidth]{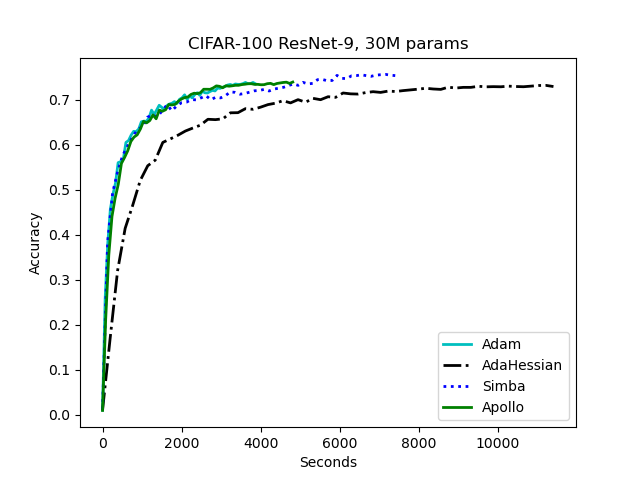}  
\end{subfigure}
\caption{Convergence behaviour and wall-clock time of various algorithms on CIFAR100 using ResNet9.}
\label{fig: resnet9}
\end{figure}

\end{document}